\documentclass[12pt, reqno]{amsart}
\usepackage{amsmath, amsthm, amscd, amsfonts, amssymb, graphicx, color}

 \usepackage{setspace}
	\usepackage{graphicx}
	 
	\usepackage{mathrsfs}

	\usepackage{graphicx} 
 
\usepackage{url}
\usepackage{enumerate}
\usepackage[bookmarksnumbered, colorlinks, plainpages]{hyperref}
\hypersetup{colorlinks=true,linkcolor=red, anchorcolor=green,
citecolor=cyan, urlcolor=red, filecolor=magenta, pdftoolbar=true}

\newtheorem{theorem}{Theorem}[section]
\newtheorem{lemma}[theorem]{Lemma}

\newtheorem{corollary}[theorem]{Corollary}
\theoremstyle{definition}

\theoremstyle{remark}
\newtheorem{remark}[theorem]{Remark}
\numberwithin{equation}{section}

  \DeclareMathOperator{\spe}{sp}

  \def\etal{et al.\,}
\begin{document}
\setcounter{page}{1}

 \title[Improvements of Some Numerical radius inequalities]
{Improvements of Some Numerical radius inequalities}
\author[M.W. Alomari]{Mohammad  W. Alomari}
 
\address{Department of Mathematics, Faculty of Science and Information
	Technology, Irbid National University, 2600 Irbid 21110, Jordan.}
\email{\textcolor[rgb]{0.00,0.00,0.84}{mwomath@gmail.com}}

\date{\today}


 
\date{Received: \today }

\begin{abstract}
In this work,  we improve and refine some numerical radius inequalities. In particular, for all Hilbert space operators $T$,   the celebrated Kittaneh inequality reads:
\begin{align*}
\frac{1}{4}\left\|   T^*T +  TT^*\right\|\le w^{2 }\left(T \right)   \le \frac{1}{2}\left\|   T^*T +  TT^*\right\|.
\end{align*}
In this work we provide some important refinements for the upper bound of the Kittaned inequality. Indeed, we establish 
\begin{align*}
w^{2 }\left(T \right)   \le \frac{1}{2}\left\|   T^*T +  TT^*\right\| 
- \frac{1}{4} \mathop {\inf }\limits_{\left\| x \right\| = 1} 
\left( {\left\langle {\left| T \right|x,x} \right\rangle  - \left\langle {\left| T^* \right|x,x} \right\rangle   } \right)^2,
\end{align*}
which also refined and improved as
\begin{align*}
w^{2 }\left(T \right)   \le \frac{1}{2}\left\|   T^*T +  TT^*\right\| 
- \frac{1}{2} \mathop {\inf }\limits_{\left\| x \right\| = 1} 
\left( {\left\langle {\left| T \right|x,x} \right\rangle  - \left\langle {\left| T^* \right|x,x} \right\rangle   } \right)^2, 
\end{align*}
and
\begin{align*}
w^{2 }\left(T \right)   \le \frac{1}{2} \left\|T^*T+TT^*  \right\|  -\frac{1}{2} 
\mathop {\inf }\limits_{\left\| x \right\| = 1} \left(\left\langle {\left| T \right|^{2 }x,x} \right\rangle^{\frac{1}{2}} - \left\langle {\left| T^* \right|^{2 } x,x} \right\rangle^{\frac{1}{2}}\right)^2, 
\end{align*}

with third improvement
\begin{align*}
w^2 \left( {T  } \right)  \le \frac{1}{4  }\left\| {\left| T \right|   + \left| {T^* } \right|  } \right\|^{2}  
- \frac{1}{{4 }}\mathop {\inf }\limits_{\left\| x \right\| = 1} \left( {\left\langle {\left| T \right|  x,x} \right\rangle    - \left\langle {\left| {T^* } \right|  x,x} \right\rangle  } \right)^2.
\end{align*}
Other general related results are also considered. 
\end{abstract}

\maketitle

\section{Introduction}

Let $\mathscr{B}\left( \mathscr{H}\right) $ be the Banach algebra
of all bounded linear operators defined on a complex Hilbert space
$\left( \mathscr{H};\left\langle \cdot ,\cdot \right\rangle
\right)$  with the identity operator  $1_\mathscr{H}$ in
$\mathscr{B}\left( \mathscr{H}\right) $. A bounded linear operator
$A$ defined on $\mathscr{H}$ is selfadjoint if and only if $
\left\langle {Ax,x} \right\rangle \in \mathbb{R}$ for all $x\in
\mathscr{H}$.  Consider the real vector space
$\mathscr{B}\left( \mathscr{H}\right)_{sa}$ of self-adjoint
operators on $ \mathscr{H}$ and its positive cone
$\mathscr{B}\left( \mathscr{H}\right)^{+}$ of positive operators
on $\mathscr{H}$.   A partial order is naturally
equipped on $\mathscr{B}\left( \mathscr{H}\right)_{sa}$ by
defining $A\le B$ if and only if $B-A\in   \mathscr{B}\left(
\mathscr{H}\right)^{+}$.  We write $A
> 0$ to mean that $A$ is a strictly positive operator, or
equivalently, $A \ge 0$ and $A$ is invertible.

The Schwarz inequality for positive operators reads that if $A$ is a positive operator in $\mathscr{B}\left(\mathscr{H}\right)$, then
\begin{align}
\left| {\left\langle {Ax,y} \right\rangle} \right|  ^2  \le \left\langle {A x,x} \right\rangle \left\langle { A y,y} \right\rangle  \label{eq1.1}
\end{align}
for any   vectors $x,y\in \mathscr{H}$.

In 1951, Reid \cite{R} proved an inequality which in some senses
considered a variant of the Schwarz inequality. In fact, he proved
that for all operators $A\in \mathscr{B}\left( \mathscr{H}\right)
$ such that $A$ is positive and $AB$ is selfadjoint then
\begin{align}
\left| {\left\langle {ABx,y} \right\rangle} \right|  \le \|B\|
\left\langle {A x,x} \right\rangle, \label{eq1.2}
\end{align}
for all $x\in \mathscr{H}$. In \cite{H}, Halmos presented his
stronger version of the Reid inequality \eqref{eq1.2} by replacing
$r\left(B\right)$ instead of $\|B\|$.

In 1952, Kato  \cite{TK} introduced a companion inequality of
\eqref{eq1.1}, called  the mixed Schwarz inequality,  which
asserts
\begin{align}
\left| {\left\langle {Ax,y} \right\rangle} \right|  ^2  \le \left\langle {\left| A \right|^{2\alpha } x,x} \right\rangle \left\langle {\left| {A^* } \right|^{2\left( {1 - \alpha } \right)} y,y} \right\rangle, \qquad 0\le \alpha \le 1. \label{eq1.3}
\end{align}
for every   operators $A\in \mathscr{B}\left( \mathscr{H}\right) $ and any vectors $x,y\in \mathscr{H}$, where  $\left|A\right|=\left(A^*A\right)^{1/2}$.

In 1988,  Kittaneh  \cite{FK4} proved  a very interesting extension combining both the Halmos--Reid inequality \eqref{eq1.2} and the   mixed Schwarz inequality \eqref{eq1.3}. His result reads that
\begin{align}
\left| {\left\langle {ABx,y} \right\rangle } \right| \le r\left(B\right)\left\| {f\left( {\left| A \right|} \right)x} \right\|\left\| {g\left( {\left| {A^* } \right|} \right)y} \right\|\label{kittaneh.ineq}
\end{align}
for any   vectors $x,y\in  \mathscr{H} $, where $A,B\in \mathscr{B}\left( \mathscr{H}\right)$ such that $|A|B=B^*|A|$ and    $f,g$ are  nonnegative continuous functions  defined on $\left[0,\infty\right)$ satisfying that $f(t)g(t) =t$ $(t\ge0)$.       Clearly, choose $f(t)=t^{\alpha}$ and $g(t)=t^{1-\alpha}$ with   $B=1_{\mathscr{H}}$ we refer to \eqref{eq1.3}. Moreover, choosing $\alpha=\frac{1}{2}$ some manipulations refer to the Halmos version of the Reid inequality. The cartesian decomposition form of \eqref{kittaneh.ineq} was recently proved by the Alomari in \cite{alomari1}.

In 1994, Furuta \cite{Furuta} proved the following generalization of Kato's inequality \eqref{eq1.3} 
\begin{align}
\left| {\left\langle {T\left| T \right|^{\alpha  + \beta  - 1} x,y} \right\rangle } \right|^2  \le \left\langle {\left| T \right|^{2\alpha } x,x} \right\rangle \left\langle {\left| T \right|^{2\beta } y,y} \right\rangle \label{eq1.5}
\end{align}
for any $x, y \in \mathscr{H}$ and $\alpha,\beta\in \left[0,1\right]$ with $\alpha+\beta \ge1$.

The inequality \eqref{eq1.5} was generalized for any $\alpha, \beta \ge0$ with $\alpha+\beta \ge1$ by Dragomir in \cite{D3}. Indeed, as noted by Dragomir the condition $\alpha,\beta\in \left[0,1\right]$ was assumed by Furuta to fit with the Heinz--Kato inequality, which reads:
\begin{align*}
\left| {\left\langle {Tx,y} \right\rangle } \right| \le \left\| {A^\alpha  x} \right\|\left\| {B^{1 - \alpha } y} \right\|
\end{align*}
for any $x, y \in \mathscr{H}$ and $\alpha \in \left[0,1\right]$  where $A$ and $B$ are prositive operators such that $\left\| {Tx} \right\| \le \left\| {Ax} \right\|$ and $\left\| {T^* y} \right\| \le \left\| {By} \right\|$ for any  $x, y \in \mathscr{H}$.

For a bounded linear operator $T$ on a Hilbert space
$\mathscr{H}$, the numerical range $W\left(T\right)$ is the image
of the unit sphere of $\mathscr{H}$ under the quadratic form $x\to
\left\langle {Tx,x} \right\rangle$ associated with the operator.
More precisely,
\begin{align*}
W\left( T \right) = \left\{ {\left\langle {Tx,x} \right\rangle :x
	\in \mathscr{H},\left\| x \right\| = 1} \right\}
\end{align*}
Also, the numerical radius is defined to be
\begin{align*}
w\left( T \right) = \sup \left\{ {\left| \lambda\right|:\lambda
	\in W\left( T \right) } \right\} = \mathop {\sup }\limits_{\left\|
	x \right\| = 1} \left| {\left\langle {Tx,x} \right\rangle }
\right|.
\end{align*}

The spectral radius of an operator $T$ is defined to be
\begin{align*}
r\left( T \right) = \sup \left\{ {\left| \lambda\right|:\lambda
	\in \spe\left( T \right) } \right\}.
\end{align*}

We recall that,  the usual operator norm of an operator $T$ is
defined to be
\begin{align*}
\left\| T \right\| = \sup \left\{ {\left\| {Tx} \right\|:x \in
	H,\left\| x \right\| = 1} \right\}.
\end{align*}

It is well known that $w\left(\cdot\right)$ defines an operator
norm on $\mathscr{B}\left( \mathscr{H}\right) $ which is
equivalent to operator norm $\|\cdot\|$. Moreover, we have
\begin{align}
\frac{1}{2}\|T\|\le w\left(T\right) \le \|T\|\label{eq1.6}
\end{align}
for any $T\in \mathscr{B}\left( \mathscr{H}\right)$ and this
inequality is sharp.

In 2003, Kittaneh \cite{FK4}  refined the right-hand side of
\eqref{eq1.7}, where he proved that
\begin{align}
w\left(T\right) \le
\frac{1}{2}\left(\|T\|+\|T^2\|^{1/2}\right)\label{eq1.7}
\end{align}
for any  $T\in \mathscr{B}\left( \mathscr{H}\right)$.

After that in 2005, the same author in \cite{FK2} proved that
\begin{align}
\frac{1}{4}\|A^*A+AA^*\|\le  w^2\left(A\right) \le
\frac{1}{2}\|A^*A+AA^*\|.\label{eq1.8}
\end{align}
The inequality is sharp.

In 2007, Yamazaki \cite{Y} improved \eqref{eq1.8} by proving that
\begin{align*}
w\left( T \right) \le \frac{1}{2}\left( {\left\| T \right\| +
	w\left( {\widetilde{T}} \right)} \right) \le \frac{1}{2}\left(
{\left\| T \right\| + \left\| {T^2 } \right\|^{1/2} }
\right)\label{eq1.10}
\end{align*}
where $\widetilde{T}=|T|^{1/2}U|T|^{1/2}$ and $U$ is the unitary operator in the polar decomposition $T$ of the form $T=U\left|T\right|$.

In 2008, Dragomir \cite{D2} used Buzano inequality to improve
\eqref{eq1.1}, where he proved that
\begin{align*}
w^2\left( T \right) \le \frac{1}{2}\left( {\left\| T \right\| +
	w\left( {T^2} \right)} \right) 
\end{align*}
This result was also recently generalized by Sattari \etal in
\cite{SMY} and Alomari in \cite{alomari2}. For more recent results about the numerical radius see the recent monograph study \cite{D1}.

In this work,  we improve and refine some numerical radius inequalities. In particular, for all Hilbert space operators $T$,   the celebrated Kittaneh inequality reads:
\begin{align*}
\frac{1}{4}\left\|   T^*T +  TT^*\right\|\le w^{2 }\left(T \right)   \le \frac{1}{2}\left\|   T^*T +  TT^*\right\|.
\end{align*}
In this work we provide some important refinements for the upper bound of the Kittaned inequality. Indeed, we establish 
\begin{align*}
w^{2 }\left(T \right)   \le \frac{1}{2}\left\|   T^*T +  TT^*\right\| 
- \frac{1}{4} \mathop {\inf }\limits_{\left\| x \right\| = 1} 
\left( {\left\langle {\left| T \right|x,x} \right\rangle  - \left\langle {\left| T^* \right|x,x} \right\rangle   } \right)^2,
\end{align*}
which also refined and improved as
\begin{align*}
w^{2 }\left(T \right)   \le \frac{1}{2}\left\|   T^*T +  TT^*\right\| 
- \frac{1}{2} \mathop {\inf }\limits_{\left\| x \right\| = 1} 
\left( {\left\langle {\left| T \right|x,x} \right\rangle  - \left\langle {\left| T^* \right|x,x} \right\rangle   } \right)^2, 
\end{align*}
and
\begin{align*}
w^{2 }\left(T \right)   \le \frac{1}{2} \left\|T^*T+TT^*  \right\|  -\frac{1}{2} 
\mathop {\inf }\limits_{\left\| x \right\| = 1} \left(\left\langle {\left| T \right|^{2 }x,x} \right\rangle^{\frac{1}{2}} - \left\langle {\left| T^* \right|^{2 } x,x} \right\rangle^{\frac{1}{2}}\right)^2, 
\end{align*}

with third improvement
\begin{align*}
w^2 \left( {T  } \right)  \le \frac{1}{4  }\left\| {\left| T \right|   + \left| {T^* } \right|  } \right\|^{2}  
- \frac{1}{{4 }}\mathop {\inf }\limits_{\left\| x \right\| = 1} \left( {\left\langle {\left| T \right|  x,x} \right\rangle    - \left\langle {\left| {T^* } \right|  x,x} \right\rangle  } \right)^2.
\end{align*}
Other general related results are also considered.

\section{Numerical Radius Inequalities}\label{sec2}

 In order to prove our main result we need to the following Lemmas:
\begin{lemma}\label{lemma2.1} 
Let $S\in \mathscr{B}\left(\mathscr{H}\right)$, $S\geq 0$ and $x\in \mathscr{H}$ be a unit vector. Then, the operator Jensen's inequality 
\begin{align}
\langle Sx, x \rangle^r \leq \langle S^rx, x \rangle,\qquad   r\ge1 
\end{align}
and 
\begin{align}
\langle S^rx, x \rangle \leq \langle Sx, x \rangle^r, \qquad   r \in \left[0,1\right].\label{eq2.2}
\end{align}
\end{lemma}
 
 Kittaneh and  Manasrah \cite{KM} obtained the following result which is a refinement of the scalar Young  inequality. 
 \begin{lemma}\label{lemma2.2}
 	Let $a, b\geq 0,$  and $p,q>1$ such that $\frac{1}{p}+\frac{1}{q}=1.$ Then 
 \begin{align}ab+\mbox{min}\left\{ \frac{1}{p},\frac{1}{q} \right\}(a^{\frac{p}{2}}-b^{\frac{q}{2}})^2 \leq \frac{a^p}{p}+\frac{b^q}{q}. \label{eq2.3}
\end{align}
 \end{lemma}
 Recently, Sheikhhosseini {\it et al.} \cite{SMS} have obtained the following  generalization of  \eqref{eq2.3}. 
 \begin{lemma}\label{lemma2.3}
 	If $a,b>0$, and $p,q>1$ such that $\frac{1}{p}+\frac{1}{q}=1,$ then for $m=1,2,3,\dots,$
 	\begin{align}
 	(a^{\frac{1}{p}}b^{\frac{1}{q}})^m+r_0^m(a^{\frac{m}{2}}-b^{\frac{m}{2}})^2 \leq \left(\frac{a^r}{p}+\frac{b^r}{q}\right)^{\frac{m}{r}},~r\geq 1,\label{eq2.4}
 	\end{align}
 	 where $ r_0=\mbox{min}\left\{ \frac{1}{p},\frac{1}{q} \right\}$. In particular, if $p=q=2$, then 
 	$$(a^{\frac{1}{2}}b^{\frac{1}{2}})^m+\frac{1}{2^m}(a^{\frac{m}{2}}-b^{\frac{m}{2}})^2 \leq 2^{\frac{-m}{r}} \left(a^r+b^r\right)^{\frac{m}{r}}.$$ 
 	For $m=1$
 	$$(a^{\frac{1}{2}}b^{\frac{1}{2}})+\frac{1}{2}(a^{\frac{1}{2}}-b^{\frac{1}{2}})^2 \leq 2^{\frac{-1}{r}} \left(a^r+b^r\right)^{\frac{1}{r}}.$$ 
 	
 \end{lemma}

In what follows, we establish some numerical radius inequalities by providing some refinements of well-known numerical radius inequalities. Let us begin with the following result.

\begin{theorem}
\label{thm2.4}Let $T\in \mathscr(B)\left(\mathscr{H}\right)$, $\alpha, \beta \ge 0$ such that $\alpha+ \beta \ge1$. Then
\begin{align}
\label{eq2.5}w^m \left( {T\left| T \right|^{\alpha  + \beta  - 1} } \right) &\le \frac{1}{2 ^{\frac{m}{r}}  }\left\| {\left| T \right|^{2r\alpha }  + \left| {T^* } \right|^{2r\beta } } \right\|^{\frac{m}{r}}  
\\
&\qquad- \frac{1}{{2^m }}\mathop {\inf }\limits_{\left\| x \right\| = 1} \left( {\left\langle {\left| T \right|^{2\alpha } x,x} \right\rangle ^{\frac{m}{2}}  - \left\langle {\left| {T^* } \right|^{2\beta } x,x} \right\rangle ^{\frac{m}{2}} } \right)^2 \nonumber
\end{align}
\end{theorem}
\begin{proof}
Let $y=x$ in \eqref{eq1.5}, then for all $m\ge1$ we have
\begin{align*}
\left| {\left\langle {T\left| T \right|^{\alpha  + \beta  - 1} x,x} \right\rangle } \right|^m  &\le \left\langle {\left| T \right|^{2\alpha } x,x} \right\rangle ^{\frac{m}{2}} \left\langle {\left| {T^* } \right|^{2\beta } x,x} \right\rangle ^{\frac{m}{2}}  
\\ 
&\le \left( {\frac{{\left\langle {\left| T \right|^{2\alpha } x,x} \right\rangle ^r  + \left\langle {\left| {T^* } \right|^{2\beta } x,x} \right\rangle ^r }}{2}} \right)^{\frac{m}{r}} \qquad (\text{by \eqref{eq2.4}})
\\
&\qquad - \frac{1}{{2^m }}\left( {\left\langle {\left| T \right|^{2\alpha } x,x} \right\rangle ^{\frac{m}{2}}  - \left\langle {\left| {T^* } \right|^{2\beta } x,x} \right\rangle ^{\frac{m}{2}} } \right)^2  
\\ 
&\le \left( {\frac{{\left\langle {\left| T \right|^{2r\alpha } x,x} \right\rangle  + \left\langle {\left| {T^* } \right|^{2r\beta } x,x} \right\rangle }}{2}} \right)^{\frac{m}{r}}  \qquad (\text{by (2.1)})
\\
&\qquad - \frac{1}{{2^m }}\left( {\left\langle {\left| T \right|^{2\alpha } x,x} \right\rangle ^{\frac{m}{2}}  - \left\langle {\left| {T^* } \right|^{2\beta } x,x} \right\rangle ^{\frac{m}{2}} } \right)^2  
\end{align*}	
Taking the supremum over all unit vector $x\in \mathscr{H}$ we get the desiredd result.
\end{proof}	 

 \begin{corollary}
\label{cor2.5} 	Let $T\in \mathscr(B)\left(\mathscr{H}\right)$, $\alpha, \beta \ge 0$ such that $\alpha+ \beta \ge1$. Then
 	\begin{align}
 \label{eq2.6}	w^2 \left( {T\left| T \right|^{\alpha  + \beta  - 1} } \right) &\le \frac{1}{2 ^{\frac{2}{r}}  }\left\| {\left| T \right|^{2r\alpha }  + \left| {T^* } \right|^{2r\beta } } \right\|^{\frac{2}{r}}  
 	\\
 	&\qquad- \frac{1}{{4 }}\mathop {\inf }\limits_{\left\| x \right\| = 1} \left( {\left\langle {\left| T \right|^{2\alpha } x,x} \right\rangle    - \left\langle {\left| {T^* } \right|^{2\beta } x,x} \right\rangle  } \right)^2 \nonumber
 	\end{align}
 \end{corollary}
\begin{proof}
Setting $m=1$ in \eqref{eq2.5} we get the desired result.
\end{proof}

\begin{remark}
Setting $r=1$ in \eqref{eq2.6}, we get
 	\begin{align*}
w^2 \left( {T\left| T \right|^{\alpha  + \beta  - 1} } \right) &\le \frac{1}{4  }\left\| {\left| T \right|^{2 \alpha }  + \left| {T^* } \right|^{2 \beta } } \right\|^{2}  
\\
&\qquad- \frac{1}{{4 }}\mathop {\inf }\limits_{\left\| x \right\| = 1} \left( {\left\langle {\left| T \right|^{2\alpha } x,x} \right\rangle    - \left\langle {\left| {T^* } \right|^{2\beta } x,x} \right\rangle  } \right)^2 \nonumber
\end{align*}
for all $\alpha, \beta \ge 0$ such that $\alpha+ \beta \ge1$.

Choosing $\alpha=\beta=\frac{1}{2}$ we get
 	\begin{align*}
w^2 \left( {T  } \right) &\le \frac{1}{4  }\left\| {\left| T \right|   + \left| {T^* } \right|  } \right\|^{2}  
 - \frac{1}{{4 }}\mathop {\inf }\limits_{\left\| x \right\| = 1} \left( {\left\langle {\left| T \right|  x,x} \right\rangle    - \left\langle {\left| {T^* } \right|  x,x} \right\rangle  } \right)^2.
\end{align*}
However, if one choose $\alpha=\beta=1$, we get
 	\begin{align*}
w^2 \left( {T\left| T \right| } \right) &\le \frac{1}{4  }\left\| {\left| T \right|^{2  }  + \left| {T^* } \right|^{2  } } \right\|^{2}  
\\
&\qquad- \frac{1}{{4 }}\mathop {\inf }\limits_{\left\| x \right\| = 1} \left( {\left\langle {\left| T \right|^{2  } x,x} \right\rangle    - \left\langle {\left| {T^* } \right|^{2  } x,x} \right\rangle  } \right)^2 \nonumber
\end{align*}
or it can be written as 	
\begin{align*}
w^2 \left( {T\left| T \right| } \right)  \le \frac{1}{4  }\left\| {T^*T+ TT^* } \right\|^{2}  
 - \frac{1}{{4 }}\mathop {\inf }\limits_{\left\| x \right\| = 1}  \left\langle {\left[ T^*T- TT^* \right] x,x} \right\rangle^2 \nonumber
\end{align*}
\end{remark}.

A generalization of the above results could be embodied as follows: 
\begin{theorem}
\label{thm2.7} 	Let $T\in \mathscr(B)\left(\mathscr{H}\right)$, $\alpha, \beta \ge 0$ such that $\alpha+ \beta \ge1$. Then
\begin{align}
w^{2s}\left( T\left| T \right|^{\alpha  + \beta  - 1} \right) &\le 2^{ - \frac{2}{r}} \left\| { \left| T \right|^{2rs\alpha}+\left| T^* \right|^{2rs\beta}   } \right\|^{\frac{2}{r}}  \nonumber 
\\ 
&\qquad-  \frac{1}{4}\mathop {\inf }\limits_{\left\| x \right\| = 1} \left[ {\left\langle {\left| T \right|^{2sr\alpha}x,x} \right\rangle  - \left\langle {\left| T^* \right|^{2rs\beta}y,y} \right\rangle } \right] \label{eq2.7}	
\end{align}	
for all $r,s\ge1$.
\end{theorem}

\begin{proof}
Let $y=x$ in \eqref{eq1.5}, by applying Lemma \ref{lemma2.3}  with $p=q=2$ and $m=2$,  we get	
\begin{align*}
& \left| {\left\langle {T\left| T \right|^{\alpha  + \beta  - 1}x,x} \right\rangle } \right|^{2s}
\\
&\le  \left\langle {\left| T \right|^{2\alpha}x,x} \right\rangle^s \left\langle {\left| T^* \right|^{2\beta} x,x} \right\rangle^s \qquad\qquad  \qquad  \text{($t^s$ increasing)}
\\ 
&\le \left\langle { \left| T \right|^{2s\alpha} x,x} \right\rangle \left\langle {  \left| {T^* } \right|^{2s\beta}  x,x} \right\rangle  \qquad\qquad  \qquad  \text{(by convexity of $t^s$)}
\\ 
&\le 2^{ - \frac{2}{r}} \left( {\left\langle {\left| T \right|^{2s\alpha} x,x} \right\rangle ^r  + \left\langle {\left| T^* \right|^{2s\beta} x,x} \right\rangle ^r } \right)^{\frac{2}{r}}   \qquad\text{(by Lemma \ref{lemma2.3})}
\\ 
&\qquad- \frac{1}{4}\left[ {\left\langle {\left| T \right|^{2sr\alpha}x,x} \right\rangle  - \left\langle {\left| T^* \right|^{2rs\beta}x,x} \right\rangle } \right] 
\\ 
&\le 2^{ - \frac{2}{r}} \left( {\left\langle {\left| T \right|^{2rs\alpha} x,x} \right\rangle    + \left\langle {\left| T^* \right|^{2rs\beta} x,x} \right\rangle   } \right)^{\frac{2}{r}}   \qquad\text{(by (2.1))}
\\ 
&\qquad-  \frac{1}{4}\left[ {\left\langle {\left| T \right|^{2sr\alpha}x,x} \right\rangle  - \left\langle {\left| T^* \right|^{2rs\beta}x,x} \right\rangle } \right].
\end{align*}
Taking the supremum over all unit vector $x\in \mathscr{H}$ we get the desiredd result.
\end{proof}

 \begin{corollary}
 \label{cor2.8}	Let $T\in \mathscr(B)\left(\mathscr{H}\right)$, $\alpha, \beta \ge 0$ such that $\alpha+ \beta \ge1$. Then
\begin{multline}
 w^{2s}\left( T\left| T \right|^{\alpha  + \beta  - 1} \right) 
 \\
  \le \frac{1}{4}\left\| { \left| T \right|^{2s\alpha}+\left| T^* \right|^{2s\beta}   } \right\|^{2}  - \frac{1}{4}\mathop {\inf }\limits_{\left\| x \right\| = 1} \left[ {\left\langle {\left| T \right|^{2s \alpha}x,x} \right\rangle  - \left\langle {\left| T^* \right|^{2 s\beta}x,x} \right\rangle } \right] 	\label{eq2.8}
 \end{multline}	
 for all $s\ge1$.
 \end{corollary}
\begin{proof}
Setting $r=1$ in \eqref{eq2.7}
\end{proof}
\begin{remark}
Setting $\alpha=\beta=\frac{1}{2}$ in \eqref{eq2.8} we get
\begin{align*}
w^{2s}\left( T  \right)  \le \frac{1}{4}\left\| { \left| T \right|^{s}+\left| T^* \right|^{s}   } \right\|^{2}   - \frac{1}{4}\mathop {\inf }\limits_{\left\| x \right\| = 1} \left[ {\left\langle {\left| T \right|^{s}x,x} \right\rangle  - \left\langle {\left| T^* \right|^{s}x,x} \right\rangle } \right] 	
\end{align*}	
for all $s\ge1$. In particular case, choose $s=1$ we get
\begin{align*}
w^{2}\left( T  \right)  \le \frac{1}{4}\left\| { \left| T \right| +\left| T^* \right|    } \right\|^{2}   - \frac{1}{4}\mathop {\inf }\limits_{\left\| x \right\| = 1} \left[ {\left\langle {\left| T \right| x,x} \right\rangle  - \left\langle {\left| T^* \right| x,x} \right\rangle } \right]. 	
\end{align*}	
\end{remark}

.\begin{remark}
Setting $\alpha=\beta=\frac{1}{s}$, $(s\ge1)$	
\begin{align}
w^{2s}\left( T\left| T \right|^{\frac{2}{s}  - 1} \right) 
\le \frac{1}{4}\left\| { \left| T \right|^{2 }+\left| T^* \right|^{2 }   } \right\|^{2}  - \frac{1}{4}\mathop {\inf }\limits_{\left\| x \right\| = 1} \left[ {\left\langle {\left| T \right|^{2 }x,x} \right\rangle  - \left\langle {\left| T^* \right|^{2 }x,x} \right\rangle } \right]. 	\label{eq2.9}
\end{align}	
In particular case, choose $s=1$ in \eqref{eq2.9} we get
\begin{align}
w^{2}\left( T\left| T \right| \right) 
\le \frac{1}{4}\left\| { \left| T \right|^{2 }+\left| T^* \right|^{2 }   } \right\|^{2}  - \frac{1}{4}\mathop {\inf }\limits_{\left\| x \right\| = 1} \left[ {\left\langle {\left| T \right|^{2 }x,x} \right\rangle  - \left\langle {\left| T^* \right|^{2 }x,x} \right\rangle } \right] 	\label{eq2.10}
\end{align}	
which can be written as
\begin{align*}
w^{2 }\left(T\left| T \right| \right)   \le \frac{1}{4}\left\|   T^*T +  TT^*\right\|^2 
- \frac{1}{4}\mathop {\inf }\limits_{\left\| x \right\| = 1} \left[ {\left\langle {\left| T \right|^{2 }x,x} \right\rangle  - \left\langle {\left| T^* \right|^{2 }x,x} \right\rangle } \right],   
\end{align*}
\end{remark}

\begin{remark}
Setting $\alpha=\beta=\frac{1}{2}$ , $s=1$, $r=2$ and 	
\begin{align*}
w^{2}\left( T  \right)  \le \frac{1}{2} \left\| { \left| T \right|^{2}+\left| T^* \right|^{2}   } \right\|    -  \frac{1}{4}\mathop {\inf }\limits_{\left\| x \right\| = 1} \left[ {\left\langle {\left| T \right|^{2}x,x} \right\rangle  - \left\langle {\left| T^* \right|^{2}x,x} \right\rangle } \right] 	
\end{align*}	
or we can write
\begin{align}
w^{2 }\left(T \right)   \le \frac{1}{2}\left\|   T^*T +  TT^*\right\|  
 -  \frac{1}{4}\mathop {\inf }\limits_{\left\| x \right\| = 1} \left[ {\left\langle {\left| T \right|^{2}x,x} \right\rangle  - \left\langle {\left| T^* \right|^{2}x,x} \right\rangle } \right] \label{eq2.11}	 
\end{align}
	and this refines the upper bound in the Kittaneh inequality \eqref{eq1.7}.
\end{remark}

\begin{theorem}
	\label{thm2.5}	Let $T\in \mathscr(B)\left(\mathscr{H}\right)$, $\alpha, \beta \ge 0$ such that $\alpha+ \beta \ge1$. Then
	\begin{align}
	\label{eq2.12} w^{2s}\left(T\left| T \right|^{\alpha  + \beta  - 1} \right)   &\le \left\|   \frac{1}{p}\left| T\right|^{2sp\alpha} +\frac{1}{q}\left| T^*\right|^{2sq\beta}\right\| 
	\\ 
	&\qquad- r_0 \mathop {\inf }\limits_{\left\| x \right\| = 1} 
	\left( {\left\langle {\left| T \right|^{2s\alpha}x,x} \right\rangle ^{\frac{p}{2}}  - \left\langle {\left| T^* \right|^{2s\beta}x,x} \right\rangle ^{\frac{q}{2}} } \right)^2 \nonumber  
	\end{align}
	for all $s\ge1$ and $p,q>1$ such that $\frac{1}{p}+\frac{1}{q}=1$, where $r_0:=\min\left\{\frac{1}{p},\frac{1}{q}\right\}$. 
	
	In particular case, we have
	\begin{align}
	\label{eq2.13}w^{2s}\left(T\left| T \right|^{\alpha  + \beta  - 1} \right)   &\le \frac{1}{2}\left\|   \left| T\right|^{4s\alpha} +\left| T^*\right|^{4s\beta}\right\| 
	\\ 
	&\qquad- \frac{1}{2} \mathop {\inf }\limits_{\left\| x \right\| = 1} 
	\left( {\left\langle {\left| T \right|^{2s\alpha}x,x} \right\rangle  - \left\langle {\left| T^* \right|^{2s\beta}x,x} \right\rangle   } \right)^2 \nonumber  
	\end{align}
\end{theorem}

\begin{proof}
	Let $s\ge1$ and setting $y=x$ in \eqref{eq1.5}, we get	
	\begin{align*}
	\left| {\left\langle {T\left| T \right|^{\alpha  + \beta  - 1}x,x} \right\rangle } \right|^{2s}
	&\le  \left\langle {\left| T \right|^{2\alpha}x,x} \right\rangle^s \left\langle {\left| T^* \right|^{2\beta} x,x} \right\rangle^s    \qquad\qquad \text{(by \eqref{eq1.5})} 
	\\ 
	&\le \left\langle {\left| T \right|^{2s\alpha}x,x} \right\rangle \left\langle {\left| T^* \right|^{2s \beta}  x,x} \right\rangle\qquad \qquad \text{(by convexity of $t^s$)}
	\\ 
	&\le \frac{1}{p}\left\langle {\left| T \right|^{2s\alpha} x,x} \right\rangle ^p  + \frac{1}{q}\left\langle {\left| T^* \right|^{2s\beta}x,x} \right\rangle ^q   \qquad \text{(by Lemma \ref{lemma2.2})}
	\\ 
	&\qquad- r_0 \left( {\left\langle {\left| T \right|^{2s\alpha}x,x} \right\rangle ^{\frac{p}{2}}  - \left\langle {\left| T^* \right|^{2s\beta}x,x} \right\rangle ^{\frac{q}{2}} } \right)^2 
	\\ 
	&\le \frac{1}{p}\left\langle {\left| T\right|^{2sp\alpha}x,x} \right\rangle  + \frac{1}{q}\left\langle { \left| T^* \right|^{2sq\beta}x,x} \right\rangle  \qquad\qquad \text{(by (2.1))}
	\\ 
	&\qquad- r_0 \left( {\left\langle {\left| T \right|^{2s\alpha}x,x} \right\rangle ^{\frac{p}{2}}  - \left\langle {\left| T^* \right|^{2s\beta}x,x} \right\rangle ^{\frac{q}{2}} } \right)^2  
	\end{align*}
	Taking the supremum over all univt vector $x\in \mathscr{H}$, we get the required result. The particular case follows by setting $p=q=2$.		
\end{proof}
Various interesting special cases could be deduced form \eqref{eq2.5}, in what follows, we give some of these cases in the consequence remarks.
\begin{remark}
	Setting $\alpha=\beta=\frac{1}{2}$ in \eqref{eq2.6}, then we have
	\begin{align*}
	w^{2s}\left(T \right)   \le \frac{1}{2}\left\|   \left| T\right|^{2s} +\left| T^*\right|^{2s }\right\| 
	- \frac{1}{2} \mathop {\inf }\limits_{\left\| x \right\| = 1} 
	\left( {\left\langle {\left| T \right|^{s}x,x} \right\rangle  - \left\langle {\left| T^* \right|^{s}x,x} \right\rangle   } \right)^2  \end{align*}
	for all $s\ge1$. In particular, for $s=1$ we get
	\begin{align*}
	w^{2}\left(T \right)    \le \frac{1}{2}\left\|   \left| T\right|^{2} +\left| T^*\right|^{2}\right\| 
	- \frac{1}{2} \mathop {\inf }\limits_{\left\| x \right\| = 1} 
	\left( {\left\langle {\left| T \right|x,x} \right\rangle  - \left\langle {\left| T^* \right|x,x} \right\rangle   } \right)^2,  
	\end{align*}
	which can be written as
	\begin{align}
	w^{2 }\left(T \right)   \le \frac{1}{2}\left\|   T^*T +  TT^*\right\| 
	- \frac{1}{2} \mathop {\inf }\limits_{\left\| x \right\| = 1} 
	\left( {\left\langle {\left| T \right|x,x} \right\rangle  - \left\langle {\left| T^* \right|x,x} \right\rangle   } \right)^2. \label{eq2.14}
	\end{align}
	and this refines the upper bound of the refinement of Kittaneh inequality \eqref{eq2.11}. Clearly, \eqref{eq2.14} is better than \eqref{eq2.11} which in turn bettern that \eqref{eq1.7}.
\end{remark}

\begin{remark}
	Setting $\alpha=\beta=1$ in \eqref{eq2.12}, then we have
	\begin{align*}
	w^{2s}\left(T\left| T \right| \right)   &\le \left\|   \frac{1}{p}\left| T\right|^{2sp} +\frac{1}{q}\left| T^*\right|^{2sq}\right\| 
	\\ 
	&\qquad- r_0 \mathop {\inf }\limits_{\left\| x \right\| = 1} 
	\left( {\left\langle {\left| T \right|^{2s}x,x} \right\rangle ^{\frac{p}{2}}  - \left\langle {\left| T^* \right|^{2s}x,x} \right\rangle ^{\frac{q}{2}} } \right)^2 \nonumber  
	\end{align*}
	for all $s\ge1$ and   $p,q>1$ such that $\frac{1}{p}+\frac{1}{q}=1$, where $r_0:=\min\left\{\frac{1}{p},\frac{1}{q}\right\}$. 
	
	In particular case, choose $s=1$ and $p=q=2$  in the previous inequality, we get
	\begin{align*}
	w^{2}\left(T\left| T \right| \right)   &\le \frac{1}{2}\left\|   \left| T\right|^{4} + \left| T^*\right|^{4}\right\| 
	- \frac{1}{2} \mathop {\inf }\limits_{\left\| x \right\| = 1} 
	\left( {\left\langle {\left| T \right|^{2}x,x} \right\rangle   - \left\langle {\left| T^* \right|^{2}x,x} \right\rangle   } \right)^2. 
	\end{align*}
\end{remark}

 Numerical radius inequality of special type of Hilbert space operators
 for commutators can be established as follows:
\begin{theorem}
\label{thm2.15}	Let $T,S\in \mathscr(B)\left(\mathscr{H}\right)$, $\alpha, \beta,\gamma, \delta \ge 0$ such that $\alpha+ \beta \ge1$ and $\gamma+ \delta \ge1$. Then
\begin{align}
\label{eq2.15}&w\left(T\left| T \right|^{\alpha  + \beta  - 1}+S\left| S \right|^{\gamma+\delta  - 1}\right)
\\
&\le  2^{-\frac{1}{r}} \left\|\left| T \right|^{2r\alpha}+\left| T^* \right|^{2r\beta}\right\|^{\frac{1}{r}}+2^{-\frac{1}{r}} \left\|\left| S \right|^{2r\gamma}+\left|S^* \right|^{2r\delta}\right\|^{\frac{1}{r}}
 \nonumber\\
&\qquad-\frac{1}{2} 
\mathop {\inf }\limits_{\left\| x \right\| = 1} \left(\left\langle {\left| T \right|^{2\alpha}x,x} \right\rangle^{\frac{1}{2}} - \left\langle {\left| T^* \right|^{2\beta} x,x} \right\rangle^{\frac{1}{2}}\right)^2
 \nonumber\\
&\qquad-\frac{1}{2}\mathop {\inf }\limits_{\left\| x \right\| = 1} 
\left(\left\langle {\left| S \right|^{2\gamma}x,x} \right\rangle^{\frac{1}{2}} - \left\langle {\left| S^* \right|^{2\delta} x,x} \right\rangle^{\frac{1}{2}}\right)^2  \nonumber
\end{align}
for all $r\ge1$.
\end{theorem}
 
\begin{proof}
Employing the triangle inequality, we have
\begin{align*}
&\left| {\left\langle {\left(T\left| T \right|^{\alpha  + \beta  - 1}+S\left| S \right|^{\gamma+\delta  - 1}\right)x,x} \right\rangle } \right|
\\
&\le  \left| {\left\langle { T\left| T \right|^{\alpha  + \beta  - 1} x,x} \right\rangle } \right|+\left| {\left\langle {S\left| S \right|^{\gamma+\delta  - 1}x,x} \right\rangle } \right|
\\
&\le  \left\langle { \left| T \right|^{2\alpha}x,x} \right\rangle ^{\frac{1}{2}} \left\langle {\left| T^* \right|^{2\beta} x,x} \right\rangle ^{\frac{1}{2}}  
 + \left\langle {\left| S \right|^{2\gamma}x,x} \right\rangle ^{\frac{1}{2}} \left\langle {\left| S^* \right|^{2\delta} x,x} \right\rangle ^{\frac{1}{2}} \qquad (\text{by  \eqref{eq1.5}}) 
\\ 
&\le 2^{-\frac{1}{r}} \left(\left\langle {\left| T  \right|^{2\alpha}x,x} \right\rangle^r + \left\langle {\left| T^* \right|^{2\beta} x,x} \right\rangle^r \right)^{\frac{1}{r}}\qquad \qquad \qquad\qquad(\text{by Lemma \ref{lemma2.3}})
\\
&\qquad\qquad-\frac{1}{2} \left(\left\langle {\left| T \right|^{2\alpha}x,x} \right\rangle^{\frac{1}{2}} - \left\langle {\left| T^* \right|^{2\beta} x,x} \right\rangle^{\frac{1}{2}}\right)^2 
\\
&\qquad + 2^{-\frac{1}{r}} \left(\left\langle {\left| S \right|^{2\gamma}x,x} \right\rangle^r + \left\langle {\left| S^* \right|^{2\delta} x,x} \right\rangle^r \right)^{\frac{1}{r}}
\\
&\qquad\qquad-\frac{1}{2} \left(\left\langle {\left| S \right|^{2\gamma}x,x} \right\rangle^{\frac{1}{2}} - \left\langle {\left| S^* \right|^{2\delta}x,x} \right\rangle^{\frac{1}{2}}\right)^2 
\\ 
&\le 2^{-\frac{1}{r}} \left(\left\langle {\left| T \right|^{2r\alpha}x,x} \right\rangle  + \left\langle {\left| T^* \right|^{2r\beta}x,x}\right\rangle \right)^{\frac{1}{r}} \qquad\qquad\qquad\qquad (\text{by (2.1)})
\\
&\qquad\qquad-\frac{1}{2} \left(\left\langle {\left| T \right|^{2\alpha}x,x} \right\rangle^{\frac{1}{2}} - \left\langle {\left| T ^*\right|^{2\beta} x,x} \right\rangle^{\frac{1}{2}}\right)^2
\\
&\qquad + 2^{-\frac{1}{r}} \left(\left\langle {\left| S \right|^{2r\gamma}x,x} \right\rangle  + \left\langle {\left| S^* \right|^{2r\delta}x,x}\right\rangle  \right)^{\frac{1}{r}}
\\
&\qquad-\frac{1}{2} \left(\left\langle {\left| S \right|^{2\gamma}x,x} \right\rangle^{\frac{1}{2}} - \left\langle {\left| S ^*\right|^{2\delta} x,x} \right\rangle^{\frac{1}{2}}\right)^2. 
\end{align*}
Taking the supremum over all unit vector $x\in \mathscr{H}$ we get the desiredd result.
\end{proof}

\begin{corollary}
	Let $T,S\in \mathscr(B)\left(\mathscr{H}\right)$, $\alpha, \beta,\gamma, \delta \ge 0$ such that $\alpha+ \beta \ge1$ and $\gamma+ \delta \ge1$. Then
\begin{align}
\label{eq2.16} w\left(T\left| T \right|^{\alpha  + \beta  - 1}+S\left| S \right|^{\gamma+\delta  - 1}\right)
 &\le  \frac{1}{2} \left\|\left| T \right|^{2 \alpha}+\left| T^* \right|^{2 \beta}  + \left| S \right|^{2 \gamma}+\left|S^* \right|^{2 \delta}\right\| 
 \\
&\qquad-\frac{1}{2} 
\mathop {\inf }\limits_{\left\| x \right\| = 1} \left(\left\langle {\left| T \right|^{2\alpha}x,x} \right\rangle^{\frac{1}{2}} - \left\langle {\left| T ^*\right|^{2\beta} x,x} \right\rangle^{\frac{1}{2}}\right)^2
\nonumber\\
&\qquad-\frac{1}{2}\mathop {\inf }\limits_{\left\| x \right\| = 1} 
\left(\left\langle {\left| S \right|^{2\gamma}x,x} \right\rangle^{\frac{1}{2}} - \left\langle {\left| S^* \right|^{2\delta} x,x} \right\rangle^{\frac{1}{2}}\right)^2.  \nonumber
\end{align}
\end{corollary}
\begin{proof}
Seeting $r=1$ in the proof of Theorem \ref{thm2.15}, and then take the supremum over all unit vector $x\in \mathscr{H}$ we get the desired result. 
\end{proof}

\begin{remark}
Setting $\alpha=\beta=\gamma=\delta=\frac{1}{2}$ in \eqref{eq2.16}, we get
	\begin{align*}
	w\left(T +S \right)
	&\le  \frac{1}{2} \left\|\left| T \right| +\left| T^* \right| + \left| S \right|+\left|S^* \right| \right\| 
 -\frac{1}{2} 
	\mathop {\inf }\limits_{\left\| x \right\| = 1} \left(\left\langle {\left| T \right| x,x} \right\rangle^{\frac{1}{2}} - \left\langle {\left| T^* \right|  x,x} \right\rangle^{\frac{1}{2}}\right)^2
	\nonumber\\
	&\qquad-\frac{1}{2}\mathop {\inf }\limits_{\left\| x \right\| = 1} 
	\left(\left\langle {\left| S \right| x,x} \right\rangle^{\frac{1}{2}} - \left\langle {\left| S^* \right|  x,x} \right\rangle^{\frac{1}{2}}\right)^2  \nonumber
	\end{align*}
In particular, take $S=T$ we get 
	\begin{align*}
w\left(T  \right)
\le  \frac{1}{2} \left\|\left| T \right| +\left| T^* \right| +   \right\| 
-\frac{1}{2} 
\mathop {\inf }\limits_{\left\| x \right\| = 1} \left(\left\langle {\left| T \right| x,x} \right\rangle^{\frac{1}{2}} - \left\langle {\left| T^* \right|  x,x} \right\rangle^{\frac{1}{2}}\right)^2
\end{align*}
\end{remark}

\begin{remark}
	Setting $\alpha=\beta=\gamma=\delta=1$ in \eqref{eq2.16}, we get
	\begin{align*}
	w\left(T\left| T \right| +S\left| S \right| \right)
	&\le  \frac{1}{2} \left\|\left| T \right|^{2 }+\left| T^* \right|^{2  }  + \left| S \right|^{2  }+\left|S^* \right|^{2 }\right\| 
	\\
	&\qquad-\frac{1}{2} 
	\mathop {\inf }\limits_{\left\| x \right\| = 1} \left(\left\langle {\left| T \right|^{2 }x,x} \right\rangle^{\frac{1}{2}} - \left\langle {\left| T^* \right|^{2 } x,x} \right\rangle^{\frac{1}{2}}\right)^2
	\nonumber\\
	&\qquad-\frac{1}{2}\mathop {\inf }\limits_{\left\| x \right\| = 1} 
	\left(\left\langle {\left| S \right|^{2 }x,x} \right\rangle^{\frac{1}{2}} - \left\langle {\left| S ^*\right|^{2 } x,x} \right\rangle^{\frac{1}{2}}\right)^2  \nonumber
	\end{align*}
In particular, take $S=T$, we get	
	\begin{align*}
w\left(T\left| T \right| \right)
&\le  \frac{1}{2} \left\|\left| T \right|^{2 }+\left| T^* \right|^{2  }  \right\|  -\frac{1}{2} 
\mathop {\inf }\limits_{\left\| x \right\| = 1} \left(\left\langle {\left| T \right|^{2 }x,x} \right\rangle^{\frac{1}{2}} - \left\langle {\left| T^* \right|^{2 } x,x} \right\rangle^{\frac{1}{2}}\right)^2
\\
&=\frac{1}{2} \left\|T^*T+TT^*  \right\|  -\frac{1}{2} 
\mathop {\inf }\limits_{\left\| x \right\| = 1} \left(\left\langle {\left| T \right|^{2 }x,x} \right\rangle^{\frac{1}{2}} - \left\langle {\left| T^* \right|^{2 } x,x} \right\rangle^{\frac{1}{2}}\right)^2
 \nonumber
\end{align*}
\end{remark}


\bibliographystyle{amsplain}

\begin{thebibliography}{99}

 \bibitem{alomari1} M.W. Alomari, On the generalized mixed Schwarz inequality, {\em Proceedings of the Institute of Mathematics and Mechanics}, National Academy of Sciences of Azerbaijan, to appear

 \bibitem{alomari2} M.W. Alomari, Refinements of some numerical radius inequalities for Hilbert space operators, {\em Linear and Multilinear Algebra}, (2019),  DOI: \url{10.1080/03081087.2019.1624682}
 
		
		
 \bibitem{D1}S.S. Dragomir, Inequalities for the numerical radius of linear operators in Hilbert spaces, SpringerBriefs in Mathematics, 2013.

\bibitem{D2}S.S. Dragomir, Some inequalities for the norm and the numerical radius of linear operator in Hilbert spaces,   Tamkang J. Math., {\bf39 } (1) (2008), 1--7.

\bibitem{D3} S.S. Dragomir, Some Inequalities generalizing Kato's and Furuta's results, {\em FILOMAT}, {\bf 28} (1) (2014), 179--195.

 
 
\bibitem{Furuta}T. Furuta, An extension of the Heinz--Kato theorem, {\em  Proc. Amer. Math. Soc.}, {\bf 120} (1994), no. 3, 785--787.

\bibitem{H} P.R. Halmos,  A Hilbert space problem book, Van Nostrand Company, Inc., Princeton,
N.J., 1967.
 

\bibitem{FK1}F. Kittaneh, M.S. Moslehian and T. Yamazaki, Cartesian decomposition and numerical radius inequalities, {\em Linear Algebra Appl.} {\bf 471} (2015), 46--53.

\bibitem{FK2}F. Kittaneh, Numerical radius inequalities for Hilbert space operators,    Studia Math., {\bf 168} (1) (2005), 73--80.

\bibitem{KM}F. Kittaneh and Y. Manasrah, Improved Young and Heinz inequalities for matrices, {\em J. Math. Anal. Appl.} {\bf 361} (2010), 262--269.

\bibitem{FK4} F. Kittaneh,  A numerical radius inequality and an estimate for the numerical radius of the Frobenius companion matrix,  Studia Math., {\bf158} (2003), 11--17.

 

\bibitem{FK7}F. Kittaneh,   Notes on some inequalities for Hilbert Space operators,   Publ. Res. Inst.
	Math. Sci., {\bf24} (2) (1988), 283--293.

\bibitem{TK}T. Kato,   Notes on some inequalities for linear operators,   Math. Ann.,  {\bf125} (1952),
208--212.



\bibitem{LD}C.-S. Lin and S.S. Dragomir, \textit{ On High-power Operator inequalities and spectral radii of operators},  Publ. RIMS, Kyoto Univ.,
{\bf42}  (2006), 391--397.


 


\bibitem{R} W. Reid,  \textit{  Symmetrizable completely continuous linear tarnsformations in Hilbert
space},   Duke Math.,  {\bf18} (1951), 41--56.

\bibitem{SMY}M. Sattari, M.S. Moslehian and T.  Yamazaki,  \textit{Some genaralized numerical radius inequalities
for Hilbert space operators},   Linear Algebra Appl., {\bf470} (2014), 1--12.

\bibitem{SMS}A. Sheikhhosseini, M. S. Moslehian and K. Shebrawi, Inequalities for generalized Euclidean
operator radius via Young’s inequality, J. Math. Anal. Appl. 445 (2017), 1516–1529.

\bibitem{Y} T. Yamazaki, \textit{On upper and lower bounds of the numerical radius and an equality condition},  Studia Math., {\bf178} (2007), 83--89.



\end{thebibliography}

\end{document}